\documentclass[11pt]{amsart}
\usepackage{amssymb,latexsym,amsmath}
\usepackage[mathscr]{eucal}

\AtBeginDocument{%
   \def\MR#1{}
}

\usepackage{amssymb,latexsym,amsmath,amsthm}
\usepackage[mathscr]{eucal}

\newtheorem{theorem}{Theorem}[section]
\newtheorem{lemma}[theorem]{Lemma}
\newtheorem{proposition}[theorem]{Proposition}

\theoremstyle{definition}
\newtheorem{remark}[theorem]{Remark}

\numberwithin{equation}{section}

\newcommand{\RR}{\ensuremath{\mathbb{R}}}

\newcommand{\prtl}{\ensuremath{\partial}}

\newcommand{\supp}{\ensuremath{\text{supp}}}

\newcommand{\wtM}{\widetilde{M}}

\newcommand{\Rtube}{{{\mathcal T}_R(\tilde \gamma)}}
\newcommand{\csch}{\text{csch}\,}
\newcommand{\mbm}{\mathbf{m}}

\title[Geodesic restrictions in nonpositive curvature]{On logarithmic improvements of critical geodesic restriction bounds in the presence of nonpositive curvature}

\author[M. D. Blair]{Matthew D. Blair}
\address{Department of Mathematics and Statistics, University of New Mexico, Albuquerque, NM 87131, USA}
\email{blair@math.unm.edu}

\begin{document}
\begin{abstract}
We consider upper bounds on the growth of $L^p$ norms of restrictions of eigenfunctions and quasimodes to geodesic segments in a nonpositively curved manifold in the high frequency limit.  This sharpens results of Chen and Sogge as well as Xi and Zhang, which showed that the crux of the problem is to establish bounds on the mixed partials of the distance function on the covering manifold restricted to geodesic segments.  The innovation in this work is the development of a formula for the third variation of arc length on the covering manifold, which allows for a coordinate free expressions of these mixed partials.
\end{abstract}
\maketitle
\section{Introduction}\label{S:intro}
Let $(M,g)$ be a compact Riemannian manifold of dimension $n$ and $\Delta_g$ the associated nonpositive Laplace operator.  It is well known that the spectrum of $-\Delta_g$ is discrete in this case, meaning there exists a sequence of eigenfunctions $\{ e_{\lambda_j}\}_{j=1}^\infty$ forming a basis for $L^2(M) = L^2(M,dV_g)$ satisfying
$$
-\Delta_g e_{\lambda} = \lambda^2e_{\lambda}
$$
with $\lambda \geq 0$ (so that $\lambda$ is an eigenvalue of $\sqrt{-\Delta_g}$).  This paper also considers spectral clusters or quasimodes, which we define as the range of the spectral projector $\mathbf{1}_{[\lambda,\lambda+f(\lambda)]}(\sqrt{-\Delta_g})$ where $f:[0,\infty)\to (0,\infty)$ is nonincreasing.  They are approximations to eigenfunctions in that
\[
\left\|(\lambda -\sqrt{-\Delta_g})\circ\mathbf{1}_{[\lambda,\lambda+f(\lambda)]}(\sqrt{-\Delta_g})\right\|_{L^2(M) \to L^2(M)} \lesssim f(\lambda),
\]
and hence the composition loses only by $f(\lambda)$.  Note that we consider exact eigenfunctions to be quasimodes since
\begin{equation}\label{projectorasidentity}
\mathbf{1}_{[\lambda,\lambda+f(\lambda)]}(\sqrt{-\Delta_g})e_\lambda = e_\lambda.
\end{equation}

A topic of considerable interest has been to study the growth of $L^p$ norms of these eigenfunctions and quasimodes in the high frequency limit as $\lambda \to \infty$ for $2 \leq p \leq \infty$.  They are viewed as a means of measuring the phase space concentration attributes of these modes.  One of the seminal works is that of Sogge \cite{sogge88}, who proved bounds on quasimodes with $f(\lambda)\equiv 1$,
\begin{equation}\label{soggebound}
\left\|\mathbf{1}_{[\lambda,\lambda+f(\lambda)]}(\sqrt{-\Delta_g})\right\|_{L^2(M)\to L^p(M)}\lesssim \lambda^{\delta(p,n)},\qquad \lambda \geq 1, \;p \in [2,\infty]
\end{equation}
with $\delta(p,n) = \max(\frac{n-1}{2}(\frac 12 - \frac 1p), \frac{n-1}{2}-\frac np )$.  These bounds give the sharp exponents for eigenfunctions on the round sphere $\mathbb{S}^n$.  Indeed, the zonal harmonics show that one cannot improve the exponent $\frac{n-1}{2}-\frac np $ and the highest weight spherical harmonics show that the exponent $\frac{n-1}{2}(\frac 12 - \frac 1p)$.  Moreover, on any $(M,g)$, these bounds are sharp for quasimodes in the range of $\mathbf{1}_{[\lambda,\lambda+1]}(\sqrt{-\Delta_g})$, which can be seen by the construction of modes with the same phase space concentration characteristics as these families of harmonics.  However, in general, it is not expected that the exponent in \eqref{soggebound} is sharp for eigenfunctions unless $(M,g)$ possesses certain dynamical features similar to that of the round sphere.  It is thus interesting to determine geometric or dynamical conditions which show that the exponent $\delta(p,n)$ in \eqref{soggebound} can be improved for eigenfunctions or quasimodes.  However, the examples referenced here mean that we must take $f(\lambda)$ to be $o(1)$ as $\lambda \to \infty$.

The works \cite{BGTrestr}, \cite{Hu}, \cite{reznikov2004norms}, \cite{ChenSogge} considered the similar problem of obtaining bounds on restrictions of eigenfunctions to geodesic segments.  For such problems, we limit the discussion to $n=2,3$.  Given a unit speed geodesic segment $\gamma:[0,1]\to M$, denote the restriction map as $\mathcal{R}_\gamma f := f|_{\gamma}$.  These works showed that for some implicit constant independent of $\gamma$,
\begin{equation}\label{restrictionbound}
\left\|\mathcal{R}_\gamma\circ\mathbf{1}_{[\lambda,\lambda+f(\lambda)]}(\sqrt{-\Delta_g})\right\|_{L^2(M)\to L^p(\gamma)}\lesssim \lambda^{\kappa(p,n)}, \quad \lambda \geq 1, \;p \in [2,\infty]
\end{equation}
where $\gamma$ is endowed with arc length measure and $\kappa(p,2) = \max(\frac 14, \frac 12-\frac 1p)$ and $\kappa(p,3) = 1-1/p$.  Moreover, these exponents are sharp for eigenfunctions on the round sphere and quasimodes in the range of $\mathbf{1}_{[\lambda,\lambda+1]}(\sqrt{-\Delta_g})$ for arbitrary $(M,g)$ by the same constructions.

When $p=\infty$ the exponent in \eqref{soggebound} is $\delta(\infty,n) = \frac{n-1}{2}$, and in this case the bounds in fact follow from pointwise Weyl laws.  As a consequence of the work of \cite{Berard}, it was shown that when the sectional curvatures of $(M,g)$ are nonpositive, the bound on the remainder in the pointwise Weyl law can be improved by a factor of $1/\log \lambda$.  Hassell and Tacy \cite{HassellTacyNonpos} expanded on the analysis in this work, showing that in the presence of nonpositive curvature, one has logarithmic gains in the bounds \eqref{soggebound} for $p \in (\frac{2(n+1)}{n-1},\infty]$ in that
\[
\left\|\mathbf{1}_{[\lambda,\lambda+(\log\lambda)^{-1}]}(\sqrt{-\Delta_g})\right\|_{L^2(M)\to L^p(M)}\lesssim \frac{\lambda^{\delta(p,n)}}{\sqrt{\log \lambda}},\qquad \lambda \geq 1,
\]
In \cite{ChenTAMS}, Chen showed that similarly, one has the bound
\begin{equation*}
\left\|\mathcal{R}_\gamma\circ\mathbf{1}_{[\lambda,\lambda+(\log\lambda)^{-1}]}(\sqrt{-\Delta_g})\right\|_{L^2(M)\to L^p(\gamma)}\lesssim \frac{\lambda^{\kappa(p,n)}}{\sqrt{\log\lambda}}, \qquad \lambda \geq 1,
\end{equation*}
where $p > 4$ when $n=2$ and $p>2$ when $n=3$.  Both $\delta(p,n)$, $\kappa(p,n)$ are the same as the exponents given above.

Recent works of the author and Sogge \cite{BlairSoggeToponogov}, \cite{blair2015refined} showed that in the same context of nonpositive curvature and $f(\lambda) = (\log\lambda)^{-1}$, one can observe similar logarithmic gains in \eqref{soggebound} when $2< p < \frac{2(n+1)}{n-1}$ and in \eqref{restrictionbound} when $n=2$ and $2 < p< 4$.  The main difference is that in each case the exponent of $(\log\lambda)^{-1}$ tends to 0 as $p \nearrow\frac{2(n+1)}{n-1}$ and $p \nearrow 4$ respectively.

Despite this success, the matter of establishing improved $L^p(M)$ bounds at the so-called ``critical" exponents of $p=\frac{2(n+1)}{n-1}$ in the presence of nonpositive curvature has shown to be a subtle matter.  The only deterministic results known at the time of this writing are due to Sogge \cite{sogge2015improved}, showing there is a gain of $(\log \log \lambda)^{-2/(n+1)^2}$.  The reason $p=\frac{2(n+1)}{n-1}$ is considered to be critical is that there are a spectrum of phase space concentration scenarios which saturate the bounds \eqref{soggebound} with $f(\lambda) \equiv 1$.  For example, the families of harmonics on $\mathbb{S}^n$ referenced above both saturate this bound as do the analogous quasimodes in the range of $\mathbf{1}_{[\lambda,\lambda+1]}(\sqrt{-\Delta_g})$ for general $(M,g)$.  The same holds true for \eqref{restrictionbound} when $(p,n) = (4,2)$ or $(p,n) = (2,3)$.  Consequently, to improve upon \eqref{soggebound}, \eqref{restrictionbound} at these critical indices, one has to simultaneously rule out several scenarios for phase space concentration.

On the other hand, works of Chen and Sogge \cite{ChenSogge} and Xi and Zhang \cite{xi2016improved} obtain improved bounds in \eqref{restrictionbound} at the critical exponent $p=4$ when $n=2$ in the presence of nonpositive curvature.  The former improved the $O(\lambda^{1/4})$ bound in \eqref{restrictionbound} to $o(\lambda^{1/4})$, and the latter quantified this, proving a gain of $(\log \log\lambda)^{- 1/8}$ in the bounds for general nonpositive curvature and a $(\log \lambda)^{-1/4}$ gain for surfaces of constant curvature with $f(\lambda) = (\log \lambda)^{-1}$.  The work \cite{ChenSogge} also showed that for $(M,g)$ of constant curvature, with $(p,n) = (2,3)$, the $O(\lambda^{1/2})$ bound can be improved to $o(\lambda^{1/2})$.

The works \cite{HezariRiviere}, \cite{hezari2016quantum} obtain logarithmic gains at critical exponents in the cases above for $(M,g)$ of negative curvature, but use quantum ergodic methods and only yield an improvement for a density one subsequence.

As observed in \cite{BGTrestr}, when $M=\mathbb{T}^n$ is the flat torus, the exponent in \eqref{restrictionbound} can be taken to be much smaller for eigenfunctions $e_\lambda$ due to the fact that the $L^\infty(M)$ bounds are much stronger.  In \cite{BourgainRudnickToralInvent}, \cite{BourgainRudnickToralGAFA}, \cite{BourgainRudnickToralIJM}, the authors considered $L^p$ bounds on restrictions of eigenfunctions to more general families of submanifolds, such as those with nonvanishing curvatures, and their relations with nodal sets of eigenfunctions on $\mathbb{T}^n$.

The purpose of this work is to first to sharpen on the work of Xi and Zhang \cite{xi2016improved}, showing the $(\log \lambda)^{-1/4}$ gains they proved when $(p,n) = (4,2)$ for constant curvature extend to the case of general nonpositive curvature.  We also show that for $(M,g)$ of constant negative curvature when $n=3$, the $o(\lambda^{1/2})$ gains of \cite{ChenSogge} can be similarly quantified.  Establishing gains in these estimates where $(\log \lambda)^{-1}$ is raised to some power is significant in that it appears to the best that can be done in the current realm of wave kernel methods, since it is unclear how to understand the kernel for time scales beyond the ``Ehrenfest time" of $t \approx \log \lambda$.

\begin{theorem}\label{T:maintheorem2D}
Let $(M,g)$ be a compact Riemannian surface with sectional curvatures pinched between $-1$ and $0$.  Then for some $C$ independent of the unit length geodesic segment $\gamma$
\begin{equation}\label{main2Dbound}
\left\|\mathcal{R}_\gamma\circ\mathbf{1}_{[\lambda,\lambda+(\log\lambda)^{-1}]}(\sqrt{-\Delta_g})\right\|_{L^2(M) \to L^4(\gamma)} \leq C \frac{\lambda^{\frac 14}}{(\log\lambda)^{\frac 14}}, \quad \lambda \geq 1.
\end{equation}
\end{theorem}
\begin{theorem}\label{T:maintheorem3D}
Let $(M,g)$ be a compact Riemannian manifold with constant sectional curvatures equal to $-1$ and $\dim(M)=3$.  Then for any $\epsilon >0$, there is a constant $C_\epsilon$ independent of the unit length geodesic segment $\gamma$ such that
\begin{equation}\label{main3Dbound}
\left\|\mathcal{R}_\gamma\circ\mathbf{1}_{[\lambda,\lambda+(\log\lambda)^{-1}]}(\sqrt{-\Delta_g})\right\|_{L^2(M) \to L^2(\gamma)} \leq C_\epsilon \frac{\lambda^{\frac 12}}{(\log\lambda)^{\frac 12-\epsilon}}, \quad \lambda \geq 1.
\end{equation}
\end{theorem}

We remark that by rescaling the metric, Theorem \ref{T:maintheorem2D} applies to any nonpostively curved manifold and Theorem \ref{T:maintheorem3D} applies to any $(M,g)$ with constant negative sectional curvatures, not just those equal to $-1$.

In \S\ref{S:reductions}, we review the standard reductions in the proofs of Theorems \ref{T:maintheorem2D} and \ref{T:maintheorem3D}, which regularizes the spectral window $\mathbf{1}_{[\lambda,\lambda+(\log\lambda)^{-1}]}(\sqrt{-\Delta_g})$ and writes it as an operator valued integral involving the wave kernel. This can then be analyzed by lifting to the universal cover $(\wtM,\tilde g)$.  As in \cite{ChenSogge} and \cite{xi2016improved} matters are then reduced to obtaining estimates on the mixed partials of the Riemannian distance function on $\wtM$ restricted to geodesic segments.  In contrast to previous works, we prove the crucial bounds on this restricted distance function in a coordinate free fashion, taking a variation through geodesics and expressing the relevant partial derivatives bounds in terms of Jacobi fields in \S\ref{S:variation}.  Finally, \S\ref{S:3D} treats the special considerations for 3 dimensions, concluding the proof of Theorem \ref{T:maintheorem3D}.

\subsection*{Acknowledgements} The author thanks David Blair for taking an early look at the geometric aspects of this work in \S3. The author was supported in part by the National Science Foundation grant DMS-1301717.

\subsection*{Notation} For nonnegative numbers $A,B$, $A \lesssim B$ means that $A \leq CB$ for some constant $C$ which is uniform in that it depends only on the manifold under consideration.  In the discussions relating to Riemannian geometry, the notation draws largely from \cite{LeeRiemannian}.  In particular, given tangent vectors $X,Y$ at a point $p$ in $M$ or its universal cover, $\langle X,Y \rangle$ denotes the inner product of the two vectors determined by the metric tensor and $|X|_g$ denotes the length $\sqrt{\langle X, X\rangle}$.  Other relevant notation is reviewed in \S\ref{S:variation}.

\section{Preliminary reductions}\label{S:reductions}
Throughout the work we fix
\[
T= \log \lambda,
\]
and let $p_n$ denote the critical exponent, that is, $p_2 = 4$, $p_3 =2$.  Let $\zeta \in \mathcal{S}(\RR)$ be an even, real valued function satisfying $\zeta(\tau)>0$ for $|\tau|\leq 1$, with $\supp(\widehat{\zeta}) \subset [-\frac 12, \frac 12]$. Then define $\zeta(T(\lambda-\sqrt{-\Delta_g}))$ by the functional calculus and observe that this is invertible on the range of the spectral projector $\mathbf{1}_{[\lambda,\lambda+(\log\lambda)^{-1}]}(\sqrt{-\Delta_g})$ with
\[
\left\|\zeta(T(\lambda-\sqrt{-\Delta_g}))^{-1}\circ\mathbf{1}_{[\lambda,\lambda+(\log\lambda)^{-1}]}(\sqrt{-\Delta_g})\right\|_{L^2(M) \to L^2(M)} \lesssim 1.
\]
It thus suffices to prove Theorems \ref{T:maintheorem2D} and \ref{T:maintheorem3D} with the spectral projector replaced by $\zeta(T(\lambda-\sqrt{-\Delta_g}))$. Define $\chi(\tau) := |\zeta(\tau)|^2 \in \mathcal{S}(\RR)$.  A standard $TT^*$ duality argument shows that the desired bounds are equivalent to
\begin{equation}\label{maindualityLp}
\left\|\mathcal{R}_\gamma\circ\chi(T(\lambda-\sqrt{-\Delta_g}))\circ\mathcal{R}_\gamma^*\right\|_{L^{p_n'}(\gamma) \to L^{p_n}(\gamma)} \lesssim \frac{\lambda^{\frac {n-1}2}}{(\log\lambda)^{\delta_n}},
\end{equation}
where $p_n$ is as above and $\delta_2 = \frac 12$, $\delta_3 = 1-\epsilon$ for any fixed $\epsilon>0$.

Given the results in \cite{ChenSogge}, \cite{xi2016improved}, \cite{SoggeZelditchL4} at any two points $p,q \in M$ the kernel of $\chi(T(\lambda-\sqrt{-\Delta_g}))(p,q)$ is now well understood as an operator valued integral
\begin{multline*}
\chi(T(\lambda-\sqrt{-\Delta_g}))(p,q) = \frac{1}{2\pi T}\int e^{i\tau\lambda} \widehat{\chi}(\tau/T)e^{-i\tau\sqrt{-\Delta_g}}(p,q)\,d\tau\\
= \frac{1}{\pi T}\int e^{i\tau\lambda} \widehat{\chi}(\tau/T)\cos\left(\tau\sqrt{-\Delta_g}\right)(p,q)\,d\tau + \tilde{\chi}(T(\lambda+\sqrt{-\Delta_g}))(p,q)
\end{multline*}
where $\tilde{\chi}(T(\lambda+\sqrt{-\Delta_g}))$ is an elliptic function whose kernel is $O(\lambda^{-N})$ for any $N$ and hence can be neglected in the analysis below.  Note that the kernel of $\mathcal{R}_\gamma\circ\chi(T(\lambda-\sqrt{-\Delta_g}))\circ\mathcal{R}_\gamma^*$ as a function of $(r,s) \in [0,1]^2$ is simply $\chi(T(\lambda-\sqrt{-\Delta_g}))(\gamma(r),\gamma(s))$. Now define
\begin{align*}
K_1(p,q) &:= \frac{1}{\pi T}\int e^{i\tau\lambda} (1-\beta)(\tau) \widehat{\chi}(\tau/T)\cos\left(\tau\sqrt{-\Delta_g}\right)(p,q)\,d\tau,\\
K_0(p,q) &:= \chi(T(\lambda-\sqrt{-\Delta_g}))(p,q)-K_1(p,q)
\end{align*}
where $\beta(\tau)=0$ for $|\tau| \geq \text{inj}(M)/2$.   By appealing to results on the local analysis of small time wave kernels (see the comments in \cite[p.442]{ChenSogge} for the $n=2$ case and Theorem 1.1 in \cite{ChenSogge} for $n=3$) it is known that the contribution of $K_0$ satisfies the stronger bound
\begin{equation}\label{K0}
\left\|\int K_0(\gamma(r),\gamma(s))f(s)\,ds\right\|_{L^{p_n}_r(0,1)}
\lesssim\frac{\lambda^{\frac {n-1}2}}{\log\lambda} \,\|f\|_{L^{p_n'}_s(0,1)}.
\end{equation}

The kernel of $K_1(p,q)$ can be analyzed by lifting to the universal cover $\wtM$ of $M$ where $\wtM$ is endowed with the pullback metric $\tilde g$ under the covering map $\pi: \wtM \to M$.  By the Cartan-Hadamard theorem, $\wtM$ is diffeomorphic to $\RR^n$ and the diffeomorphism can be taken as $\exp_p: T_p M \cong \RR^n \to \wtM$ for any $p\in M$.  When $M$ has constant curvatures equal to $-1$, $(\wtM,\tilde{g})$ is isometric to hyperbolic space $\mathbb{H}^n$ (with the same curvature normalization).    We now fix a fundamental domain $D\subset \wtM$.  Let $\Gamma$ denote the group of deck transformations, the diffeomorphisms $\alpha: \wtM \to \wtM$  satisfying $\pi \circ \alpha = \pi$.  Given any $p,q \in M$, let $\tilde p,\tilde q \in D$ be the unique points such that $\pi(\tilde p)=p$, $\pi(\tilde q)=q$. As observed in \cite{SoggeZelditchL4}, we have the following formula which relates the wave kernel on $(M,g)$ to the wave kernel $\cos(\tau\sqrt{-\Delta_{\tilde g}})$ on $(\wtM, \tilde g)$:
\begin{equation}\label{cosinepropagator}
\cos\left(\tau\sqrt{-\Delta_g}\right)(p,q) = \sum_{\alpha \in \Gamma}\cos\left(\tau\sqrt{-\Delta_{\tilde g}}\right) (\alpha(\tilde p),\tilde q).
\end{equation}
By hyperbolic lattice point counting results and finite speed of propagation, it is known that there are at at most $O(e^{C\tau})$ nonzero terms in this sum.  In particular, when $\wtM$ is 3 dimensional hyperbolic space, we can take $C=2$.

We further denote $\tilde \gamma$ as the geodesic in $\wtM$ satisfying $\pi \circ \tilde\gamma(r) = \gamma(r)$.
This reduces us to considering considering kernels defined by
\[
K_\alpha(r,s) := \frac{1}{\pi T}\int_{-T}^T e^{i\tau\lambda} (1-\beta)(\tau) \widehat{\chi}\left(\frac{\tau}{T}\right)\cos\left(\tau\sqrt{-\Delta_{\tilde g}}\right)(\alpha(\tilde\gamma(r)),\tilde\gamma(s))\,d\tau,
\]
and in particular, we have $K_1 (\gamma(r),\gamma(s))= \sum_{\alpha \in \Gamma}K_\alpha(r,s)$.  Note that the restricted limits of integration are a consequence of the finite propagation speed of the wave kernel.  Moreover, there are at most $O(e^{CT})$ nonzero terms in the right hand side of this identity (again with $C=2$ in the setting of the $n=3$ theorem).  As observed in \cite{xi2016improved}, when $\alpha$ is the identity, the kernel $K_{Id}(r,s)$ can be treated in a fashion similar to \eqref{K0}, so we may restrict attention to the remaining terms.  Now set
\[
\phi_\alpha(r,s) := d(\alpha(\tilde\gamma(r)),\tilde\gamma(s))
\]
where $d(\cdot,\cdot)$ denotes Riemannian distance on $(\wtM,\tilde g)$.  Again by finite propagation speed, we may restrict attention to cases where $\phi_\alpha(r,s) \leq T$.

We now appeal to Lemma 3.1 in \cite{ChenSogge}, which shows that\footnote{There are harmless notational differences between this display and those in \cite{ChenSogge}, including dependence of various functions on $\gamma, \lambda, T$ and the normalization of $a_\pm$.}
\begin{equation}\label{asymptotics}
K_\alpha(r,s) = \frac{\lambda^{\frac {n-1}2}}{T(\phi_\alpha(r,s))^{\frac{n-1}{2}}}\sum_\pm a_{\pm,\alpha}(r,s)e^{\pm i\lambda \phi_\alpha(r,s)} + R_\alpha(r,s)
\end{equation}
where $|R_\alpha(r,s)| \lesssim e^{CT}$ for some uniform $C$, and
$
\|a_{\pm,\alpha}\|_{C^1([0,1]^2)} \lesssim 1.
$
This lemma is a consequence of the Hadamard parametrix for $\cos(\tau\sqrt{-\Delta_{\tilde g}})$ (see e.g. \cite{SoggeHangzhou}) and stationary phase.  In what follows, we harmlessly neglect the contribution of $R_\alpha(r,s)$ as the bound $e^{CT} = (\log \lambda)^C \lesssim \lambda^{\frac {n-1}2}/\log\lambda$ means the contribution of  $\sum_\alpha|R_\alpha(r,s)|$ satisfies stronger bounds.

As observed in \cite{BlairSoggeToponogov}, the following proposition is a consequence of the Toponogov comparison theorem.  In what follows, $\nabla_1 d(p,q)$ denotes the Riemannian gradient of the distance function in the first variable, which gives the unit tangent vector at $p$ of the geodesic from $q$ to $p$.

\begin{proposition}\label{T:toponogovprop}
Suppose $(\wtM,\tilde g)$, $\dim(\wtM)\geq 2$, is a simply connected manifold such that the sectional curvatures of are nonpositive and bounded below by $-1$.  Suppose that $\tilde \gamma$ is a geodesic with $\tilde\gamma(0)=p$. Given $T,\theta>0$, let $C(\theta;T)$ denote the intersection of the geodesic ball $B(p,T)$ in $\wtM$ with the cone of aperture $\theta$ about $\gamma$ with vertex $p$.  That is, $C(\theta;T)$ is equal to
\[
\{q \in \wtM: d(p,q) \leq T \text{ and } \angle(\dot{\tilde{\gamma}}(0),\pm\nabla_1 d(p,q)) \leq \theta \text{ for some choice of $\pm$}\}.
\]
Now fix $R>0$ and denote $\Rtube:=\{q \in \wtM: d(q,\tilde\gamma) \leq R\}$.  Then for $T$ sufficiently large, if $\theta_T$ is defined by $\sin(\theta_T/2) = \frac{\sinh (R/2)}{\sinh T}$ then
\[
C(\theta_T;T) \subset \Rtube.
\]
\end{proposition}

We now fix $R$ sufficiently large so that $\theta_T \geq e^{-T}$ for every $T$.  As in \cite{xi2016improved}, we now set
\[
\Gamma_{{\mathcal T}_R(\tilde \gamma)}  = \left\{\alpha\in \Gamma: \alpha(D) \cap \Rtube \neq \emptyset\right\}
\]
and split the sum
\begin{equation}\label{staboscdecomp}
\sum_\alpha K_\alpha(r,s) = \sum_{\alpha \in \Gamma_{{\mathcal T}_R(\tilde \gamma)}}K_\alpha(r,s)+\sum_{\alpha \notin \Gamma_{{\mathcal T}_R(\tilde \gamma)}}K_\alpha(r,s).
\end{equation}
The first sum here captures the contributions of geodesic segments $\alpha(\tilde \gamma(r))$ that are in some sense close to the full geodesic $\tilde\gamma(s)$, $s \in \RR$ in the cotangent bundle.  As will be shown (following \cite{ChenSogge}, \cite{xi2016improved}), the second sum satisfies better bounds; it can be viewed as capturing the microlocal contributions of the kernel of $R_\gamma\circ\chi(T(\lambda-\sqrt{-\Delta_g}))\circ R_\gamma^*$ which are in some sense ``transverse" to the geodesic segment.  This seems to give at least a partial explanation as to why critical $L^p$ bounds on restrictions to geodesics have been a more forgiving problem than the analogous one over the full manifold: since the cotangent bundle $T^*\gamma$ is a submanifold of smaller dimension than $T^*M$, there is a microlocally transverse contribution to the relevant operators which is better behaved.  For critical $L^p(M)$ bounds, it appears to be difficult to similarly decompose the kernel of $\chi(T(\lambda-\sqrt{-\Delta_g}))$.

We recall the approach in \cite{xi2016improved} (which follows \cite{ChenSogge}, \cite{BlairSoggeToponogov}) used to handle the first sum in \eqref{staboscdecomp}.  As observed in \cite{BlairSoggeToponogov}, we have that $\phi_\alpha(r,s) \approx d(p,\alpha(p))$ where $p = \tilde\gamma(0)$ for all but finitely many $\alpha$ and that
\[
\#\{\alpha \in \Gamma_{{\mathcal T}_R(\tilde \gamma)}: d(p,\alpha(p)) \in [2^k,2^{k+1}] \} = O(2^k).
\]
The latter property is a consequence of the fact that for any geodesic ball of the fixed radius $R$, there are at most $O(1)$ translates of the $\alpha(D)$ of $D$ that can intersect an arbitrary ball of radius $R$.  But this shows that
\begin{equation}\label{stabbd}
\left|\sum_{\alpha \in \Gamma_{{\mathcal T}_R(\tilde \gamma)}}K_\alpha(r,s)\right| \lesssim \frac{\lambda^{\frac {n-1}2}}{T}\sum_{1 \lesssim 2^k \leq T} 2^{k}2^{-k(\frac{n-1}{2})} + e^{CT}.
\end{equation}
Hence by geometric summation and Young's inequality the contribution of this sum to \eqref{maindualityLp} is indeed $O(\lambda^{(n-1)/2}/(\log\lambda)^{\delta_n})$.

The crux of the proofs of Theorems \ref{T:maintheorem2D} and \ref{T:maintheorem3D} is to bound the contribution of the second sum in \eqref{staboscdecomp}.  In the rest of the paper, we restrict attention to the $\alpha \in \Gamma$ such that $\phi_\alpha(r,s)$ is sufficiently large for all $(r,s) \in [0,1]^2$.  This can be done by including the finite collection of remaining terms in the sum \eqref{stabbd}, which does not affect the validity of that bound.

We now outline the rest of the proof of Theorem \ref{T:maintheorem2D}; our three dimensional result has some special considerations which we postpone until \S\ref{S:3D}.

\subsection{Conclusion of the proof of Theorem \ref{T:maintheorem2D}}
The remaining ingredients are the following geometric lemmas, extending the bounds in \cite[Lemmas 5,6]{xi2016improved} in the work of Xi and Zhang to the more general setting of nonconstant nonpositive curvature.  We will prove them in \S \ref{S:variation}.  To state them, we let $(\wtM,\tilde{g})$ be the universal cover of a manifold with sectional curvatures pinched between -1 and 0, let $\gamma,\eta:[0,1] \to \wtM$ be disjoint unit speed geodesic segments (since we have no need to distinguish between $\tilde\gamma$ and $\gamma$ below), and define $\phi(r,s) = d(\gamma(s),\eta(r))$.  The two lemmas establish exponential lower and upper bounds on the mixed partials of $\phi$ respectively.
\begin{lemma}\label{T:mixedpartial}  Let $\wtM$, $\phi$ be as above and suppose $\dim(M)=2$.  Fix $(r_0,s_0)\in [0,1]^2$.  Denote $\rho_0 = d(\gamma(s_0),\eta(r_0))$ and suppose further that $3 \leq \rho_0 \leq T$ (so that $\gamma([0,1])\cap \eta([0,1])=\emptyset$) and $\angle (\nabla_1 d(\gamma(s_0),\eta(r_0)),\pm\dot\gamma(s_0)) \geq e^{-T}$ for both choices of $\pm$.  Then there exist uniform constants $C_i>0$, $i=1,2,3$ such that for $T$ sufficiently large, we have
\begin{equation}\label{thirdpartiallower}
|\prtl_{rss}^3\phi(r_0,s_0)| \geq e^{-C_1T},
\end{equation}
whenever
\begin{equation}\label{anglehyp}
\angle(\dot\eta(r_0),\pm\dot\sigma(0)) < e^{-C_2T}.
\end{equation}
for some choice of $\pm$.  If \eqref{anglehyp} fails to hold, we then have the lower bound
\begin{equation}\label{secondpartiallower}
|\prtl_{rs}^2\phi(r_0,s_0)| \geq e^{-C_3T}.
\end{equation}
\end{lemma}

\begin{lemma}\label{T:expupperbd}
Let $\wtM$, $\phi$ be as above, but now allow $\dim(M)\geq 2$. Then for $k\geq 2$, there exists a uniform constant $C_k >1$ (distinct from those in Lemma \ref{T:mixedpartial}) such that
\[
\sum_{|\alpha| =k } |\prtl^\alpha \phi(r_0,s_0)| \leq C_k e^{C_kT}.
\]
\end{lemma}

Returning to the proof of Theorem \ref{T:maintheorem2D}, note that we may apply the Lemma above to the case where $\tilde\gamma(s)$ is relabeled as $\gamma(s)$ and $\alpha(\tilde\gamma(r))= \eta(r)$ to analyze the phase function $\phi_\alpha(r,s)$ when $\alpha \notin \Gamma_{{\mathcal T}_R(\tilde \gamma)}$.  Indeed, by Proposition \ref{T:toponogovprop}, if $\alpha\notin \Gamma_{{\mathcal T}_R(\tilde \gamma)}$ then for both choices of $\pm$
\[
\angle\left(\dot{\tilde{\gamma}}(s),\pm\nabla_1 d\big(\tilde{\gamma}(s),\alpha(\tilde{\gamma}(r))\big)\right) \geq e^{-T}, \qquad  \text{for all }(r,s)  \in [0,1]^2.
\]

With Lemma \ref{T:mixedpartial} in hand, we now present a bound on the oscillatory integral operators of interest here.  It is a variation on \cite[Proposition 2]{xi2016improved} and based on an observation in \cite[p.56]{PhongSteinTorsion}.
\begin{lemma}\label{T:oiolemma}
Let $T_\lambda$ be an oscillatory integral operator of the form
\begin{equation}\label{oio}
T_\lambda f(r) = \int_0^1 e^{i\lambda \phi(r,s)} a(r,s) f(s)\,ds
\end{equation}
where $a \in C^1([0,1]^2)$ and $\phi \in C^4([0,1]^2)$.  Suppose further $|\prtl_{rss}^3\phi(r,s)| \geq C$ whenever $|\prtl_{rs}^2\phi(r,s)| \leq C$ and that $A\geq\max(\|\prtl_{rs}^2\phi\|_{C^2},1)$.
Then
\begin{equation}\label{tlambdabound}
\|T_\lambda f\|_{L^2(0,1)}\lesssim \lambda^{-\frac 14}C^{-1}A^{\frac 54}\|a\|_{C^1}\|f\|_{L^2(0,1)}
\end{equation}
\end{lemma}

We postpone the proof of this lemma until the end of this section and show that Lemmas \ref{T:mixedpartial}, \ref{T:expupperbd}, and \ref{T:oiolemma} conclude the proof of Theorem \ref{T:maintheorem2D}.  Taking $C = \min(e^{-C_1T},e^{-C_3T})$ where $C_1,C_3$ are as in Lemma \ref{T:mixedpartial} and $A = e^{CT}$ for some constant sufficiently large we obtain that for some larger constant $C$,
\[
\left\|\int_0^1 K_\alpha(r,s)f(s)\,ds\right\|_{L^2(0,1)} \lesssim \frac{\lambda^{\frac 14}e^{CT}}{T}\|f\|_{L^2(0,1)}.
\]
Since \eqref{asymptotics} gives the bound
\[
\left\|\int_0^1 K_\alpha(r,s)f(s)\,ds\right\|_{L^\infty(0,1)} \lesssim \frac{\lambda^{\frac 12}}{T}\|f\|_{L^1(0,1)},
\]
we may interpolate to obtain an $L^{4/3}(0,1) \to L^4(0,1)$ bound.  Since the number of nonzero terms in $\Gamma_{{\mathcal T}_R(\tilde \gamma)}$ is $O(e^{CT})$, we have for some larger $C$,
\begin{multline*}
\sum_{\alpha \notin \Gamma_{{\mathcal T}_R(\tilde \gamma)} }\left\|\int_0^1 K_\alpha(r,s)f(s)\,ds\right\|_{L^{4}(0,1)} \\ \lesssim \frac{\lambda^{\frac 38}e^{CT}}{T}\|f\|_{L^{4/3}(0,1)} \lesssim \lambda^{\frac 12}(\log\lambda)^{-\frac 12}\|f\|_{L^{4/3}(0,1)}
\end{multline*}
as $e^{CT} \lesssim \lambda^{\frac 18}$.  This concludes the proof of \eqref{maindualityLp} and Theorem \ref{T:maintheorem2D}.

\begin{proof}[Proof of Lemma \ref{T:oiolemma}]
First observe that the lower bounds on $|\prtl_{rs}^2\phi|$, $|\prtl_{rss}^3\phi|$, are stable over cubes of sidelength $\frac{C}{2A}$.  Indeed, if $|\prtl_{rs}^2\phi(r_0,s_0)| \geq C$ or $|\prtl_{rss}^3\phi(r_0,s_0)| \geq C$, then for $|r-r_0|,|s-s_0| \leq \frac{C}{2A}$, a simple Taylor expansion shows that $|\prtl_{rs}^2\phi(r,s)| \geq \frac C2$ or $|\prtl_{rss}^3\phi(r,s)| \geq \frac C2$ respectively.

Now write $[0,1] = \cup_k I_k$ where $\{I_k\}_k$ is an almost disjoint union of intervals with $|I_k| \leq \frac{C}{4A}$ and the number of intervals is less than $\frac{8A}{C}$.  let
\[
a_k :=\mathbf{1}_{I_k}(r) a(r,s),
\]
so that $a=\sum_k a_k$ almost everywhere.  Thus if $T_{\lambda,k}$ denotes the oscillatory integral operator defined by replacing $a$ by $a_k$ in \eqref{oio}, we have that
\[
\|T_\lambda\|_{L^2(0,1)\to L^2(0,1)} \lesssim \sqrt{\frac AC} \max_k \|T_{\lambda,k}\|_{L^2(0,1)\to L^2(I_k)}
\]
which by duality reduces us to showing that uniformly in $k$, we have
\begin{equation}\label{dualoiobound}
\|T_{\lambda,k}T_{\lambda,k}^*\|_{L^2(I_k)\to L^2(I_k)} \lesssim \lambda^{-\frac 12}C^{-1}A^{\frac 32}\|a\|_{C^1}^2.
\end{equation}
The kernel of the operator on the left takes the form
\[
K_k(r,r') := \int_0^1e^{i\lambda (\phi(r,s)-\phi(r',s))} a_k(r,s) \overline{a_k(r',s)} \,ds
\]
For $r \neq r'$, we rewrite this as $K_k = \sum_l K_{k,l}$ almost everywhere with
\begin{gather}
K_{k,l}(r,r'):=\int e^{i\mu\varphi(r,r',s)} \mathbf{1}_{I_l}(s)  a_k(r,s) \overline{a_k(r',s)} \,ds,\label{kldefn}\\
\varphi (r,r',s) := \frac 2C \frac{\phi(r,s)-\phi(r',s)}{r-r'},\qquad \mu := \frac C2\lambda (r-r') .\notag
\end{gather}
By the observations above, we have that either $|\prtl_s \varphi (r,r',s)|\geq 1$ throughout the support of the amplitude in \eqref{kldefn} or that $|\prtl_{ss}^2\varphi (r,r',s)|\geq 1$ throughout.

We now claim that
\begin{equation}\label{kernelklbd}
|K_{k,l}(r,r')| \lesssim \mu^{-\frac 12}\|a\|_{C^1}^2(1+\|\prtl_{rss}^3\phi\|_{L^\infty}).
\end{equation}
If this holds, then using support properties of $K_{k,l}$, this yields
\begin{align*}
\int |K_{k}(r,r')| \,dr &\leq \sum_l\int_{|r-r'|\lesssim \frac{C}{4A}} |K_{k,l}(r,r')| \\
& \lesssim AC^{-\frac 32}\lambda^{-\frac 12}\|a\|_{C^1}^2(1+\|\prtl_{rss}^3\phi\|_{L^\infty})\int_{|r-r'|\lesssim \frac{C}{4A}} |r-r'|^{-\frac 12} \,dr \\
 & \lesssim A^{\frac 32}C^{-1}\lambda^{-\frac 12}\|a\|_{C^1}^2,
\end{align*}
and since the same holds replacing $dr$ by $dr'$, \eqref{dualoiobound} will follow.

To see \eqref{kernelklbd}, first observe that it suffices to restrict attention to the case $\mu \geq 1$.  In this case it is a consequence of
\begin{multline*}
\left|\int e^{i\mu\varphi(r,r',s)} \tilde{a}_{k,l}(r,r',s)\,ds\right| \\
\lesssim \mu^{-1/2}\|\tilde{a}_{k,l}\|_{L^\infty} + \mu^{-1/2}|I_l|\left( \|\prtl_s\tilde{a}_{k,l}\|_{L^\infty} + \|\tilde{a}_{k,l}\|_{L^\infty} \|\prtl_{ss}^2\varphi\|_{L^\infty}\right)
\end{multline*}
which follows by a standard integration by parts when $|\prtl_s \varphi (r,r',s)|\geq 1$ (phase has no critical points) and a stationary phase estimate otherwise (see e.g. \cite[p.332-4]{SteinHarmonicAnalysis}).
\end{proof}

\section{The variation through geodesics}\label{S:variation}
We fix $(r_0,s_0) \in [0,1]^2$ and seek to compute $\prtl_{rs}^2\phi(r_0,s_0)$, $\prtl_{rss}^3\phi(r_0,s_0)$ and higher order derivatives of $\phi$ by using a variation through geodesics.  We thus let $\rho_0 = \phi(r_0,s_0) = d(\gamma(s_0),\eta(r_0))$ and define
\[
\Psi(r,s,t): [0,1]\times [0,1]\times [0,\rho_0] \to \wtM
\]
to be the time $t$ value of the geodesic starting at $\eta(r)$ and ending at $\gamma(s)$ parameterized with speed $d(\gamma(s),\eta(r))/\rho_0$.  Thus $t\mapsto \Psi(r_0,s_0,t)$ is parameterized by arc length, but this is not necessarily true for other $(r,s)$. Hence
\begin{equation}\label{lengthfunctional}
\phi(r,s) = d(\gamma(s),\eta(r)) =  \int_0^{\rho_0} |\prtl_t\Psi(r,s,t)|_g\,dt.
\end{equation}
We denote $\sigma(t):[0,\rho_0]\to M$ as this unit speed geodesic from $\eta(r_0)$ to $\gamma(s_0)$ so that $\sigma(t) = \Psi(r_0,s_0,t)$.  Also, $\dot{\sigma}_t$ denotes the velocity of $\sigma$ at time $t$.

In what follows, we use $\nabla$ to denote usual operations involving the Levi-Civita (Riemannian) connection and $D_r,D_s,D_t$ to denote covariant differentiation along curves with respect to $r,s,t$.  Observe that $\Psi(r,s,0)=\eta(r)$ and $\Psi(r,s,\rho_0)=\gamma(s)$. Hence $D_r \prtl_r \Psi|_{t=0}=0$, and $D_s \prtl_s \Psi|_{t=\rho_0}=0$.  Moreover,
\begin{equation}\label{vanishingendpointsprep}
\prtl_r\Psi(r,s,\rho_0)=0 \qquad \text{and} \qquad \prtl_s\Psi(r,s,0)=0.
\end{equation}
We will make frequent use, often without reference, of the so called ``symmetry lemma" (see e.g. \cite[Lemma 6.3]{LeeRiemannian} or \cite[p.68]{doCarmoRiemannian}), stating that
\[
D_r \prtl_s \Psi = D_s \prtl_r \Psi
\]
and the same holds for any choice of two variables from $(r,s,t)$.  Given these observations, we have the following properties which will be of use below:
\begin{equation}\label{vanishingendpoints}
D_s\prtl_r\Psi\big|_{t=0,\rho_0},D_r\prtl_r\Psi\big|_{t=0,\rho_0}, D_s\prtl_s\Psi\big|_{t=0,\rho_0}=0.
\end{equation}

By its very definition, $\Psi$ is a variation through geodesics in that for every fixed $(r,s)$, $t \mapsto \Psi(r,s,t)$ is a geodesic, hence $X=\prtl_s \Psi$ and $X=\prtl_r\Psi$ satisfy the Jacobi equation (see e.g. \cite[Ch. 10]{LeeRiemannian} or \cite[Ch. 5]{doCarmoRiemannian})
\begin{equation}\label{jacobiequation}
D_t^2 X + R(X,\prtl_t\Psi)\prtl_t\Psi=0
\end{equation}
where $R(X,Y)Z = \nabla_X \nabla_Y Z - \nabla_Y\nabla_X Z-\nabla_{[X,Y]}Z$ denotes the Riemann curvature endomorphism (this is consistent with the conventions in \cite{LeeRiemannian}, \cite{CheegerEbin} but not \cite{doCarmoRiemannian}).  We often use $Rm$ to denote the the curvature tensor this yields, $Rm(X,Y,Z,W):= \langle R(X,Y)Z, W\rangle$. We will be particularly interested in the values of $\prtl_s\Psi$, $\prtl_r\Psi$ when $(r,s)=(r_0,s_0)$, so we set
\[
V_t := \prtl_r\Psi(r_0,s_0,t), \qquad W_t := \prtl_s\Psi(r_0,s_0,t)
\]
Note that by construction
\begin{equation}\label{endpointinfo}
V_0 = \dot\eta(r_0),\; W_{\rho_0}= \dot\gamma(s_0)\qquad \text{ and } \qquad V_{\rho_0} = 0, \;W_0 =0
\end{equation}

Next we recall the distinguishing properties of perpendicular Jacobi fields (solutions to \eqref{jacobiequation}).  A Jacobi field $X_t$ is said to be \emph{normal} along $\sigma(t)$ if $X_t \perp \dot{\sigma}_t$ for every $t \in [0,\rho_0]$.  As in \cite[Lemma 10.6]{LeeRiemannian} (or similarly \cite[p.118-9]{doCarmoRiemannian}), $X_t$ is normal if and only if $X_a \perp \dot{\sigma}_a$ and $D_t X_t|_{t=a}\perp \dot{\sigma}_a$ for some $a \in [0,\rho_0]$ or if $X_t$ is perpendicular to $\dot{\sigma}_t$ at two points in $[0,\rho_0]$.  This can be verified by appealing to a property we use below, namely that the tangential component $\langle X_t, \dot{\sigma}_t\rangle$ of an arbitrary (not necessarily normal) Jacobi field $X_t$ is a linear function of $t$.  Indeed, this is a consequence of of the compatibility of the connection and the symmetries of $Rm$, which imply that $\langle D_tX_t, \dot{\sigma}_t\rangle$ is constant.  Hence an arbitrary Jacobi field $X_t$ decomposes into tangential and normal parts, both of which are solutions to the Jacobi equation
\[
X_t^\top:=\langle X_t,\dot{\sigma}_t \rangle \dot{\sigma}_t, \qquad X_t^\perp:=X_t-X_t^\top.
\]
In particular, this means that $D_t (X_t^\perp) = (D_tX_t)^\perp$, so that the meaning of $D_t X_t^\perp$ is unambiguous below. We remark that as a particular case of this, $\langle D_tV_t,\dot{\sigma}_t\rangle$, $\langle D_tW_t,\dot{\sigma}_t\rangle$ are constant and hence we have unambiguously
\begin{equation}\label{linearity}
\langle V_0,\dot{\sigma}_0\rangle = -\rho_0 \langle D_tV_t,\dot{\sigma}_t\rangle, \qquad \langle W_{\rho_0},\dot{\sigma}_{\rho_0}\rangle = \rho_0 \langle D_tW_t,\dot{\sigma}_t\rangle.
\end{equation}

A common method of analyzing normal Jacobi fields such as $V_t^\perp$, $W_t^\perp$ is to introduce parallel vector fields $P_t$, $Q_t$ along $\sigma(t)$ satisfying $D_tP_t =D_tQ_t = 0$ and $|P_t|_g = |Q_t|_g=1$, then writing $W_t^\perp = w(t)Q_t$ and similarly for $V_t^\perp$.  This is particularly illuminating when $\wtM$ has constant sectional curvature $\mathscr{K}\equiv -1$ or $\mathscr{K}\equiv 0$ as in this case inserting the expression $w(t)Q_t$ in \eqref{jacobiequation} yields the differential equation $w''(t) = w(t)$ and $w''(t) = 0$ respectively.  Thus if $Q_t$ is normalized to have unit length, then since $W_0=0$
\[
W_t^\perp =
\begin{cases}
\sinh t |D_t W_0^\perp|_gQ_t, & \mathscr{K} \equiv -1,\\
 t |D_t W_0^\perp|_gQ_t, & \mathscr{K} \equiv 0
\end{cases}.
\]
As a consequence of the Rauch comparison theorem (cf. \cite[Ch. 10]{doCarmoRiemannian}, \cite[Theorem 1.33]{CheegerEbin}), we have that in the hypotheses of Theorem \ref{T:maintheorem2D}, where sectional curvatures are pinched between $-1$ and $0$, $|W_t^\perp|_g$ is bounded by the lengths of these two expressions, that is,
\begin{equation}\label{jacobipinch}
t |D_t W_0^\perp|_g \leq |W_t^\perp|_g \leq \sinh t |D_t W_0^\perp|_g, \qquad t \in [0,\rho_0].
\end{equation}
Evaluating this at $t=\rho_0$ yields upper and lower bounds on $|D_t W^\perp_0|_g $
\begin{equation}\label{jacobipinchbounds}
\csch \rho_0 |\dot\gamma(s_0)^\perp|\leq |D_t W^\perp_0|_g \leq \frac{1}{\rho_0} |\dot\gamma(s_0)^\perp|.
\end{equation}
When $\mathscr{K} \equiv -1$, one can instead solve the two point boundary value problem with values given by \eqref{endpointinfo} to obtain
\begin{equation}\label{paralleljacobi}
V_t^\perp = \frac{\sinh(\rho_0-t)}{\sinh{\rho_0}}|\dot\eta(r_0)^\perp |_gP_t, \qquad W_t^\perp = \frac{\sinh t }{\sinh{\rho_0}}|\dot\gamma(s_0)^\perp|_gQ_t .
\end{equation}
We will make use of \eqref{jacobipinchbounds}, \eqref{paralleljacobi} below to prove Theorems \ref{T:maintheorem2D} and \ref{T:maintheorem3D}.

\subsection{The second and third variation formulas}
\begin{theorem}\label{T:variation}
Let $\Psi(r,s,t)$ be the variation through geodesics defined above.  Then $\phi$ satisfies the second variation formula at $(r_0,s_0)$
\begin{equation}\label{secondvariation}
\prtl_{rs}^2\phi(r_0,s_0) =-\langle V_0^\perp,D_tW_0^\perp\rangle,
\end{equation}
where $D_tW_0$ denotes $D_t W_t|_{t=0}$.  We also have the ``third variation" formula
\begin{multline}\label{thirdvariation}
\prtl_{rss}^3\phi(r_0,s_0)=
-\langle V_0^\perp,(D^2_s\prtl_t\Psi)^\perp|_{(r_0,s_0,0)}\rangle \\
+ 2\langle D_tW_0^\perp, V_0^\perp\rangle \langle D_t W_0, \dot{\sigma}_0\rangle+|D_tW_0^\perp|^2_g \langle V_0, \dot{\sigma}_0\rangle.
\end{multline}
\end{theorem}
The formula \eqref{secondvariation} is in some sense classical, implicitly appearing in texts such as \cite[\S1.7]{CheegerEbin}, \cite[Proposition 10.14]{LeeRiemannian} but we give a complete proof here in the interest of setting the stage for \eqref{thirdvariation}.

\begin{proof}
Differentiating \eqref{lengthfunctional} in $r$ and applying the symmetry lemma yields
\[
\prtl_{r}\phi(r,s) = \int_0^{\rho_0} \frac{\langle D_t \prtl_r\Psi(r,s,t), \prtl_t\Psi(r,s,t)\rangle}{|\prtl_t\Psi(r,s,t)|_g}\,dt .
\]
In what follows, we suppress the arguments of $\Psi$. Using the symmetry lemma again, we have $\prtl_{rs}^2\phi(r,s)$ is
\begin{equation*}
\int_0^{\rho_0} \frac{\langle D_s D_t \prtl_r\Psi, \prtl_t\Psi\rangle+\langle D_t \prtl_r\Psi, D_t\prtl_s\Psi\rangle}{|\prtl_t\Psi|_g}-\frac{\langle D_t \prtl_r\Psi, \prtl_t\Psi\rangle\langle D_t\prtl_s \Psi,\prtl_t\Psi\rangle}{|\prtl_t\Psi|_g^3}\,dt.
\end{equation*}
We next recall the following commutator formula, valid for any vector field $X$ along the variation in that $X(r,s,t) \in T_{\Psi(r,s,t)} \wtM$ (cf. \cite[Lemma 10.1]{LeeRiemannian} or \cite[Lemma 4.1, p.98]{doCarmoRiemannian})
\begin{equation}\label{covariantcomm}
[D_s,D_t]X := D_s D_t X- D_t D_s X =R(\prtl_s\Psi,\prtl_t\Psi)X.
\end{equation}
Applying this with $X=\prtl_r\Psi$ along with the skew symmetry of $Rm$ in the last two entries (cf. \cite[Proposition 7.4]{LeeRiemannian} or \cite[p.91]{doCarmoRiemannian}), $\prtl_{rs}^2\phi(r,s)$ is
\begin{multline}\label{fullmixedpartial}
\int_0^{\rho_0} \frac{\langle D_t D_s \prtl_r\Psi, \prtl_t\Psi\rangle +
\langle R(\prtl_s\Psi,\prtl_t\Psi)\prtl_r\Psi, \prtl_t\Psi\rangle + \langle D_t \prtl_r\Psi, D_t\prtl_s\Psi\rangle}{|\prtl_t\Psi|_g}\,dt\\
-\int_0^{\rho_0}\frac{\langle D_t \prtl_r\Psi, \prtl_t\Psi\rangle\langle D_t\prtl_s \Psi,\prtl_t\Psi\rangle}{|\prtl_t\Psi|_g^3}\,dt = \\
\left[\frac{\langle D_s \prtl_r\Psi, \prtl_t\Psi\rangle}{|\prtl_t\Psi|_g}\right]_{t=0}^{t=\rho_0} + \int_0^{\rho_0} \frac{- \langle R(\prtl_s\Psi,\prtl_t\Psi)\prtl_t\Psi, \prtl_r\Psi\rangle+
\langle D_t \prtl_r\Psi, D_t\prtl_s\Psi\rangle}{|\prtl_t\Psi|_g}\,dt\\
-\int_0^{\rho_0}\frac{\langle D_t \prtl_r\Psi, \prtl_t\Psi\rangle\langle D_t\prtl_s \Psi,\prtl_t\Psi\rangle}{|\prtl_t\Psi|_g^3}\,dt
\end{multline}
where used that for fixed $r,s$, the curve $t\mapsto \Psi(r,s,t)$ is a geodesic to obtain
\[
\frac{\prtl}{\prtl t}\left(\frac{\langle D_s \prtl_r\Psi, \prtl_t\Psi \rangle}{|\prtl_t \Psi|_g}\right) = \frac{\langle D_t D_s \prtl_r\Psi, \prtl_t\Psi \rangle}{|\prtl_t \Psi|_g}.
\]
But given \eqref{vanishingendpoints}, the boundary term in \eqref{fullmixedpartial} vanishes for all $(r,s)$.

At this stage, it is common to simplify \eqref{fullmixedpartial} to obtain
\begin{equation*}
\prtl_{rs}^2\phi(r_0,s_0) = \int_0^{\rho_0} \langle D_t V_t^\perp, D_t W_t^\perp\rangle- Rm(W_t^\perp,\dot{\sigma}_t,\dot{\sigma}_t, V_t^\perp)\,dt.
\end{equation*}
Instead, we use the Jacobi equation $D_t^2 \prtl_s \Psi = -R(\prtl_s\Psi,\prtl_t\Psi)\prtl_t\Psi$ to get that
\begin{equation*}
\prtl_{rs}^2\phi(r,s) = \int_0^{\rho_0} \frac{\prtl}{\prtl t} \left(\frac{\langle \prtl_r\Psi, D_t\prtl_s\Psi\rangle}{|\prtl_t \Psi|_g}\right) -\frac{\langle D_t \prtl_r\Psi, \prtl_t\Psi\rangle\langle D_t\prtl_s \Psi,\prtl_t\Psi\rangle}{|\prtl_t\Psi|_g^3}\,dt
\end{equation*}
Since $\prtl_r\Psi(r,s,\rho_0)\equiv 0$, we are only concerned with the contribution at $t=0$ after integrating the first term in the integrand.  Moreover, $\langle D_t \prtl_r\Psi, \prtl_t\Psi\rangle$ is independent of $t$, and the same holds when $s,\prtl_s\Psi$ replace $r,\prtl_r\Psi$.  Hence the second term here is simply $\rho_0$ times the integrand, and the factors can be evaluated at any point.  Thus since $\langle \prtl_r\Psi, \prtl_t\Psi\rangle|_{t=0}=-\rho_0\langle D_t \prtl_r\Psi, \prtl_t\Psi\rangle|_{t=\rho_0}$ by the same linearity principle as in \eqref{linearity}, we conclude that
\begin{equation}\label{simple2ndpartial}
\prtl_{rs}^2\phi(r,s) = -\frac{\langle \prtl_r\Psi, D_t\prtl_s\Psi\rangle}{|\prtl_t \Psi|_g}\bigg|_{t=0}  +
\frac{\langle \prtl_r\Psi, \prtl_t\Psi\rangle\langle D_t\prtl_s\Psi,\prtl_t\Psi\rangle}{|\prtl_t\Psi|_g^3}\bigg|_{t=0}.
\end{equation}
At $(r_0,s_0)$, this is $-\langle V_0,D_tW_0\rangle + \langle V_0,\dot\sigma_0\rangle \langle D_tW_0,\dot\sigma_0\rangle$ in the notation above since $|\prtl_t\Psi|_{(r_0,s_0)}|_g\equiv 1$.  But by orthogonal decompositions, this is \eqref{secondvariation}.

We now consider $\prtl_{rss}^3\phi(r,s)$.  Differentiating \eqref{simple2ndpartial} in $s$ yields
\begin{multline*}
\prtl_{rss}^3\phi(r,s)=-\frac{\langle \prtl_r\Psi, D_s^2\prtl_t\Psi\rangle}{|\prtl_t \Psi|_g}\bigg|_{t=0} +
\frac{\langle \prtl_r\Psi, D_t\prtl_s\Psi\rangle\langle D_t\prtl_s\Psi,\prtl_t\Psi\rangle}{|\prtl_t \Psi|_g^3}\bigg|_{t=0}\\
+\frac{\langle \prtl_r\Psi, D_t\prtl_s\Psi\rangle\langle D_t\prtl_s\Psi,\prtl_t\Psi\rangle}{|\prtl_t\Psi|_g^3}\bigg|_{t=0}
+\frac{\langle \prtl_r\Psi, \prtl_t\Psi\rangle
|D_t\prtl_s\Psi|_g^2}{|\prtl_t\Psi|_g^3}\bigg|_{t=0}
\\
+\frac{\langle \prtl_r\Psi, \prtl_t\Psi\rangle\langle D_s^2\prtl_t\Psi,\prtl_t\Psi\rangle}{|\prtl_t\Psi|_g^3}\bigg|_{t=0}
-\frac{3\langle \prtl_r\Psi, \prtl_t\Psi\rangle\langle D_t\prtl_s\Psi,\prtl_t\Psi\rangle^2}{|\prtl_t\Psi|_g^5}\bigg|_{t=0}
\end{multline*}
where we have made use of \eqref{vanishingendpoints} and repeated use of the symmetry lemma $D_t\prtl_s\Psi=D_s\prtl_t\Psi$.  The presentation here is such that the first line is the result of differentiating the first term in \eqref{simple2ndpartial}, but the second and third terms are indeed identical.  At $(r_0,s_0)$, this is
\begin{multline*}
-\langle V_0,D_s^2\prtl_t \Psi|_{(r_0,s_0,0)} \rangle
+2\langle V_0,D_tW_0\rangle \langle D_tW_0,\dot\sigma_0\rangle
+\langle V_0,\dot\sigma_0\rangle|D_tW_0|_g^2\\
+\langle V_0, \dot\sigma_0 \rangle \langle D_s^2\prtl_t\Psi|_{(r_0,s_0,0)},\dot\sigma_0\rangle
-3\langle V_0,\dot\sigma_0\rangle \langle D_tW_0,\dot\sigma_0\rangle^2
\end{multline*}
The identity \eqref{thirdvariation} follows by an orthogonal decomposition of the vectors in the top row and observing the cancellation with the bottom row.
\end{proof}

\begin{remark}\label{R:mixedpartialexpressions}
We stress that the above proof can be modified to show that
\begin{align*}
\prtl_{ss}^2\phi(r,s) &= \phantom{-}\frac{\langle D_t\prtl_s \Psi, \prtl_s \Psi\rangle}{|\prtl_t\Psi|_g}\bigg|_{t=\rho_0}+
\rho_0\frac{\langle D_t\prtl_s\Psi, \prtl_t\Psi\rangle^2 }{|\prtl_t\Psi|_g^3} \bigg|_{t=\rho_0},\\
\prtl_{rr}^2\phi(r,s) &= -\frac{\langle D_t\prtl_r \Psi, \prtl_r \Psi\rangle}{|\prtl_t\Psi|_g}\bigg|_{t=0}+
\rho_0\frac{\langle D_t\prtl_r\Psi, \prtl_t\Psi\rangle^2 }{|\prtl_t\Psi|_g^3} \bigg|_{t=0}.
\end{align*}
In the notation above, at $(r_0,s_0)$ the latter expression simplifies to
\begin{equation}\label{2ndpartialr0}
\prtl_{rr}^2\phi(r_0,s_0) = -\langle D_tV_0^\perp, V_0^\perp \rangle.
\end{equation}
These expressions will be used in the proofs of Lemmas \ref{T:expupperbd} and \ref{T:isolatedlemma} below.
\end{remark}
\subsection{Proof of Lemma \ref{T:expupperbd}}
To prove this lemma we use the identities in Remark \ref{R:mixedpartialexpressions} and \eqref{simple2ndpartial}.  Differentiating them in $r,s$ and using standard properties of the connection, the bounds in Theorem \ref{T:expupperbd} reduce to the following lemma, which also applies with $\prtl_r \Psi$ replacing $\prtl_s \Psi$.

\begin{lemma}\label{T:iteratedcovariantlemma}
Suppose $X$ is a smooth vector field along the variation $\Psi$ defined above.  Let $\mathbf m = (m_1,m_2,\dots, m_j)$ be a composition of the integer $k$ (ordered partition of the integer $k$) and let $D^\mbm X$ denote an iterated covariant differentiation of $X$ of the form
\[
\begin{cases}
D^\mbm X = D^{m_1}_r D^{m_2}_s \cdots D^{m_{j-1}}_r D^{m_{j}}_s X, & $j$ \text{ even}\\
D^\mbm X = D^{m_1}_r D^{m_2}_s \cdots D^{m_{j-1}}_s D^{m_{j}}_r X, & $j$ \text{ odd}
\end{cases}
\]
or as a similar expression with the roles of $r,s$ reversed.  Then there exists a uniform constant $C_k >0$ such that for $t \in [0,\rho_0]$
\begin{equation}\label{iteratedexpbds}
|D^\mbm \prtl_s \Psi(r_0,s_0,t)|_g + |D_t D^\mbm \prtl_s \Psi(r_0,s_0,t)|_g \lesssim e^{C_k \rho_0}.
\end{equation}
\end{lemma}
\begin{proof}
The first observation is that $D^\mbm \prtl_s \Psi$ satisfies an nonhomogeneous Jacobi equation of the form
\begin{equation}\label{nonhomogjacobi}
D_t^2 D^\mbm \prtl_s \Psi + R(D^\mbm \prtl_s \Psi, \prtl_t \Psi)\prtl_t\Psi + S_\mbm =0
\end{equation}
where $S_\mbm$ is a vector field along the variation which is induced\footnote{The relation between tensors and vector fields can be reviewed in \cite[Lemma 2.4]{LeeRiemannian}.} by the pullback of a sum of tensors on $M$, evaluated along a subcollection of the vector fields $D^{\mathbf{\tilde m}} \prtl_s \Psi$, $D_tD^{\mathbf{\tilde m}} \prtl_s \Psi$, $\prtl_t \Psi$ for which $\mathbf{\tilde m}$ is a composition of an integer strictly less than $k$.  This can be verified by induction on $k$, where the $k=0$ case is the homogeneous Jacobi equation \eqref{jacobiequation}.  We thus assume \eqref{nonhomogjacobi} holds and take the derivative $D_r$ of both sides (or similarly $D_s$)
\begin{multline}\label{jacobicommute}
0=D_t^2 D_r D^\mbm \prtl_s \Psi + R(D_r D^\mbm \prtl_s \Psi, \prtl_t \Psi)\prtl_t\Psi + [D_r,D_t^2] D^\mbm \prtl_s \Psi \\
+\big( D_r(R(D^\mbm \prtl_s \Psi, \prtl_t \Psi)\prtl_t\Psi )- R(D_r D^\mbm \prtl_s \Psi, \prtl_t \Psi)\prtl_t\Psi\big) + D_rS_\mbm
\end{multline}
The first two terms here are the principal terms in \eqref{nonhomogjacobi}.  Recalling \eqref{covariantcomm} and properties of covariant differentiation of tensors (cf. \cite[Lemma 4.6]{LeeRiemannian} or \cite[p.102]{doCarmoRiemannian}), we have the following computation in $X=D^\mbm \prtl_s \Psi$
\begin{align*}
[D_r,D_t^2] X&= R(\prtl_r \Psi, \prtl_t \Psi)D_t X+ D_t\big(R(\prtl_r \Psi, \prtl_t \Psi)X\big)\\
&=2R(\prtl_r \Psi, \prtl_t \Psi)D_t X+R(D_t \prtl_r \Psi, \prtl_t \Psi)X+ (D_tR)(\prtl_r \Psi, \prtl_t \Psi)X
\end{align*}
where the last term on the right is a vector field which can be realized as the result of raising the 1-form which acts on tangent vectors $Y$ by
\[
Y \mapsto \nabla Rm (\prtl_s \Psi, \prtl_t \Psi,X,Y,\prtl_t\Psi),
\]
where $\nabla Rm$ is the covariant derivative of the tensor.  Since the projection map $\pi : \wtM \to M$ is a local isometry, this can be realized as the pullback of a tensor on $M$.  Hence its action on vectors in the unit ball bundle of tangent vectors is uniformly bounded.  Next we turn to the term in parentheses in \eqref{jacobicommute}, which can be written similarly as
\begin{equation*}
R(D^\mbm \prtl_s \Psi, D_t \prtl_r \Psi)\prtl_t\Psi + R(D^\mbm \prtl_s \Psi, \prtl_t \Psi)D_t \prtl_r\Psi
 + (D_rR)(D^\mbm \prtl_s \Psi, \prtl_t \Psi)\prtl_t\Psi.
\end{equation*}
The last term here can similarly be expressed in terms of $\nabla Rm$.  A similar argument which takes the covariant derivative of the tensor inducing $S_\mbm$ allows us to compute the last term in \eqref{jacobicommute}.

We now use \eqref{nonhomogjacobi} at the point $(r_0,s_0)$ and induct on $k$ to establish \eqref{iteratedexpbds}. Using the Jacobi equation \eqref{jacobiequation} for the base step $k=0$, we have
\[
\frac 12 \prtl_t \left(|\prtl_s \Psi|_g^2 +  |D_t\prtl_s \Psi|_g^2  \right) = \langle \prtl_s \Psi, D_t\prtl_s \Psi \rangle - Rm(\prtl_s \Psi, \dot\sigma_t, \dot\sigma_t, D_t\prtl_s \Psi).
\]
But $Rm$ can be realized as the pullback of a tensor on $M$ and since $|\dot\sigma_t|_g=1$, there exists a uniform constant $C_0>0$ such that
\[
\frac 12 \prtl_t \left(|\prtl_s \Psi|_g^2 +  |D_t\prtl_s \Psi|_g^2  \right) \leq
C_0\left( |\prtl_s \Psi|_g^2 +  |D_t\prtl_s \Psi|_g^2  \right).
\]
Denoting $\prtl_s\Psi(r_0,s_0,t)=W_t$, by Gronwall's inequality, \eqref{endpointinfo}, and \eqref{jacobipinchbounds},
\begin{equation*}
|W_t|_g +  |D_tW_t|_g \lesssim   \left(|W_0|_g +  |D_tW_0|_g\right) e^{C_0t} \lesssim e^{C_0t}.
\end{equation*}

We now assume that we have a bound of the form \eqref{iteratedexpbds} for every integer strictly less than $k$.  Let $\mbm$ be a composition of $k$ and set $X=D^\mbm \prtl_s \Psi|_{(r_0,s_0)}$ so that $X_t=X(r_0,s_0,t)$ solves the two point boundary value problem
\begin{equation}\label{nonhomogjacobibvp}
\begin{cases}
D_t^2 X_t + R(X_t, \dot\sigma_t)\dot\sigma_t + S_\mbm =0\\
X(r_0,s_0,0)=0=X(r_0,s_0,\rho_0)
\end{cases}.
\end{equation}
Indeed, given \eqref{vanishingendpoints}, we have $D^\mbm \prtl_s \Psi|_{t=0,\rho_0} \equiv 0$.  Treating $X_t$ as a solution to a two point boundary value problem is preferable to treating it as a solution to an initial value problem as before since we do not have uniform estimates on $D_t X_t$ a priori.  We now write $X_t=Y_t+Z_t$ where $Z_t$ will be a homogeneous solution to \eqref{jacobiequation} and $Y_t$ solves the nonhomogeneous equation in \eqref{nonhomogjacobibvp} but with vanishing initial data $Y_0,D_tY_0=0$.  We now observe
\begin{align*}
\prtl_t|Y_t|_g &= |Y_t|_g^{-1} \langle D_t Y_t,Y_t \rangle \leq |D_t Y_t|_g\\
\prtl_t|D_tY_t|_g &= -|D_tY_t|_g^{-1} \left(Rm(Y_t,\prtl_t \Psi, \prtl_t \Psi, D_tY_t)+\langle S_\mbm,D_tY_t\rangle\right)\\
 & \lesssim C_0|Y_t|_g + e^{C_{k-1}\rho_0}
\end{align*}
so the integral version of Gronwall inequality applied to their sum yields
\[
|Y_t(r_0,s_0,t)|_g + |D_tY_t(r_0,s_0,t)|_g \lesssim e^{C_k\rho_0 }
\]
for some constant $C_k$.  We now take $Z_t$ to solve the homogeneous Jacobi equations with $Z_0=0$, $Z_{\rho_0}=-Y_{\rho_0}$ (which exists since $\wtM$ has no conjugate points, cf. \cite[Exercise 10.2]{LeeRiemannian}, \cite[Proposition 3.9]{doCarmoRiemannian}).  Arguing as in \eqref{jacobipinchbounds}, we have that
$
|D_tZ_0|_g \leq |Y_{\rho_0}| \lesssim e^{C_k\rho_0}
$
so by the Gronwall argument for the homogeneous equation above, we have  $|X_t|_g \lesssim e^{C_k\rho_0}$.
\end{proof}

\subsection{Proof of Lemma \ref{T:mixedpartial}}\label{SS:proofofpartiallowers}
The first ingredient we need are lower and upper bounds on $|D_tW_0^\perp|_g$.  To accomplish this we use the consequence of the Rauch comparision theorem \eqref{jacobipinchbounds}. Since $\alpha \notin \Psi_{{\mathcal T}_R(\tilde \gamma)}$, $e^{-T}\leq \angle(\dot\gamma(s_0),\dot\sigma_{\rho_0}) \leq \pi - e^{-T}$ and hence for $T$ sufficiently large
\[
|\dot\gamma(s_0)^\perp|_g = \sin\angle(\dot\gamma(s_0),\dot\sigma_{\rho_0})  \geq \frac{e^{-T}}2.
\]
We thus have that by \eqref{jacobipinchbounds},
\begin{equation}\label{deeteedublower}
e^{-2T}\leq |D_t W^\perp_0|_g \leq 1.
\end{equation}

Now suppose that \eqref{anglehyp} is satisfied, that is, $\angle(\dot\eta(r_0),\pm\dot\sigma(0)) < e^{-C_2T}$ for some choice of $\pm$, then
\[
| V_0^\perp|_g = | \dot\eta(r_0)^\perp|_g = \sin\angle(\dot\eta(r_0),\pm\dot\sigma(0)) < e^{-C_2T}.
\]
At the same time, by \eqref{covariantcomm},
\[
D_sD_t\prtl_s\Psi(r_0,s_0,0) = D_tD_s\prtl_s\Psi(r_0,s_0,0) + R(\prtl_s\Psi,\prtl_t\Psi)\prtl_s\Psi|_{(r_0,s_0,0)}
\]
and the second term here vanishes since $\prtl_s\Psi(r_0,s_0,0)=0$.  Thus by Lemma \ref{T:iteratedcovariantlemma} and the symmetry lemma, $|D_s^2\prtl_t\Psi(r_0,s_0,0)| \leq Ce^{CT}$ for some uniform constant $C$, implying that
\begin{equation}\label{accelerationbd}
|\langle V_0^\perp,(D^2_s\prtl_t\Psi)^\perp|_{(r_0,s_0,0)}\rangle| \leq Ce^{(C-C_2)T}.
\end{equation}
Moreover, by \eqref{linearity}, $|\langle D_tW_0,\dot\sigma_0\rangle| = |\langle W_{\rho_0},\dot\sigma_{\rho_0}\rangle|/\rho_0\leq 1$ and for $C_2$ and $T$ sufficiently large, we have that $|\langle V_0, \dot{\sigma}_0\rangle|\geq \frac 12$. Thus in view of \eqref{thirdvariation}, we have that for some larger value of $C$
\begin{equation*}
|\prtl_{rss}^3\phi(r_0,s_0)| \geq \frac 12|D_tW_0^\perp|^2_g -Ce^{(C-C_2)T} \geq \frac 12 e^{-4T}-Ce^{(C-C_2)T}
\end{equation*}
Thus if we take $C_2$ sufficiently large, we obtain a bound of the form \eqref{thirdpartiallower}.

Otherwise, if \eqref{anglehyp} is not satisfied, then since the tangent space at any point is 2 dimensional, whenever $V_0^\perp = \eta(r_0)^\perp$ is nonzero, it is parallel to $D_tW_0^\perp$ and hence by \eqref{deeteedublower}
\begin{equation}\label{closeanglebd}
\sin\angle(\dot\eta(r_0),\pm\dot\sigma(0)) = | V_0^\perp|_g =  \frac{|\langle V_0^\perp,D_tW_0^\perp\rangle|}{|D_tW_0^\perp|_g}
\leq e^{2T}|\prtl_{rs}^2\phi(r_0,s_0)|,
\end{equation}
showing that one can obtain a lower bound of the form \eqref{secondpartiallower}.

\subsection{Comparison with results in normal coordinates}
In \cite{ChenSogge}, the authors computed $\prtl_{rs}^2\phi(r_0,s_0)$ and $\prtl_{rss}^3\phi(r_0,s_0)$ by taking geodesic normal coordinates about $\eta(r_0)$.  In such coordinates, geodesics through the origin are straight lines parameterized with constant Euclidean velocity.  Thus if a point $q\in \wtM$ is denoted as $x$ in coordinates, $\nabla_1 d(\eta(r),q)$ is simply $-x/|x|$.  Now let $x(s)$ be the coordinate curve determined by $\gamma(s)$, and $y$ denote the tangent vector $\dot\eta(r_0)$, so that
\[
\prtl_{r}\phi(r_0,s) = \langle \nabla_1 d(\eta(r_0),\gamma(s)), \dot\eta(r_0)\rangle = -\frac{x(s)}{|x(s)|}\cdot y
\]
Here $\cdot$, $|\cdot|$ denote the Euclidean dot product and length respectively.  Indeed, the metric tensor at the origin is simply $g_{ij} = \delta_{ij}$ and the Christoffel symbols with respect to the coordinate frame vanish at the origin as well.  Differentiating $-\frac{x(s)}{|x(s)|}\cdot y$ in $s$, it can be seen that
\begin{equation}\label{secondpartialcoord}
\prtl_{rs}^2\phi(r_0,s_0) = -\frac{1}{|x(s_0)|}\left(\dot x(s_0)\cdot y- \left(\dot x(s_0)\cdot\frac{ x(s_0)}{|x(s_0)|}\right)\left(y\cdot\frac{x(s_0)}{|x(s_0)|}\right)\right)
\end{equation}
\begin{multline}\label{thirdpartialcoord}
\prtl_{rss}^3\phi(r_0,s_0)=-\frac{1}{| x(s_0)|}\left(\ddot{ x}(s_0)\cdot y- \left(\ddot{ x}(s_0)\cdot\frac{ x(s_0)}{| x(s_0)|}\right)
\left( y\cdot\frac{ x(s_0)}{| x(s_0)|}\right)\right)\\
+\frac{2 \dot x(s_0)\cdot\frac{ x(s_0)}{| x(s_0)|}}{| x(s_0)|^2}
\left(\dot x(s_0)\cdot y-\left(y\cdot \frac{ x(s_0)}{| x(s_0)|}\right) \left(\dot x(s_0)\cdot \frac{ x(s_0)}{| x(s_0)|}\right)\right)\\
+\frac{y\cdot\frac{ x(s_0)}{| x(s_0)|}}{| x(s_0)|^2}\left(|\dot x (s_0)|^2-\left(\dot x(s_0)\cdot \frac{ x(s_0)}{| x(s_0)|}\right)^2\right).
\end{multline}

It can be seen that \eqref{secondvariation} does indeed agree with \eqref{secondpartialcoord}.  To see this recall that $V_0 = \dot{\eta}(r_0)$, which we set as $y$ in coordinates.  Moreover, $\dot\sigma_0$ is $x(s_0)/|x(s_0)|$ and $\rho_0 = |x(s_0)|$ by the observations above.  Thus if we can check that $D_tW_0=\dot x(s_0)/|x(s_0)|$ in coordinates, we will have that both expressions are indeed the same.  But this follows from \cite[Lemma 10.7]{LeeRiemannian} or \cite[p.113]{doCarmoRiemannian}, which implies that in the coordinate system, the Jacobi field $W_t$ satisfying $W_0=0$ and $D_t W_0=z$ takes the form $tz$.  Thus we must have $|x(s_0)|z = \rho_0 z=\dot x(s_0)$ as these give $W_{\rho_0} = \dot\gamma(s_0)$ in coordinates.

One can similarly verify that \eqref{thirdpartialcoord} agrees with \eqref{thirdvariation}.  Here the only additional observation is that $\prtl_t \Psi(r_0,s,0)$ is $x(s)/\rho_0$ in coordinates, which implies that $D_s^2\prtl_t\Psi(r_0,s_0,0)$ is given by $ \ddot x(s_0)/|x(s_0)|$.  The first, second, and third lines in \eqref{thirdpartialcoord} correspond to the respective terms in \eqref{thirdvariation}.

It is therefore possible to obtain an analogous proof of Lemma \ref{T:mixedpartial} by using the expressions \eqref{secondpartialcoord}, \eqref{thirdpartialcoord}, reasoning similarly to \cite{ChenSogge} and \S\ref{SS:proofofpartiallowers}.  However, the coordinate free approach here has advantages in that the variation through geodesics yields the bounds in Lemma \ref{T:expupperbd} for derivatives such as $\prtl^2_{rrss}\phi$, which are needed for Lemma \ref{T:oiolemma} to be effective.

\section{Endpoint bounds in 3 dimensions}\label{S:3D}
In this section we conclude the proof of Theorem \ref{T:maintheorem3D}.   We claim that when $\alpha \notin \Gamma_{{\mathcal T}_R(\tilde \gamma)}$, there exists $C$ sufficiently large such that
\begin{equation}\label{kalphabd3D}
\left\|\int_0^{\frac 14}K_\alpha(r,s)f(s)\,ds\right\|_{L^2_r([0,\frac 14])}\lesssim \left(\frac{\lambda}{(\log \lambda)^{3}}+ \lambda^{\frac 34}e^{CT}\right)\|f\|_{L^2_s([0,\frac 14])}.
\end{equation}
The restriction to $(r,s) \in [0,1/4]^2$ here is harmless as one can write the square $[0,1]^2$ as the almost disjoint union of 16 cubes of sidelength $1/4$ and apply the argument below to each one to yield \eqref{kalphabd3D} with $[0,1]$ on both sides. Once this is established we then recall that there are $O(e^{2T})=O((\log \lambda)^2)$ nonzero kernels $K_\alpha$, and hence the proof of Theorem \ref{T:maintheorem3D} is completed as the contribution of $\alpha \in \Gamma_{{\mathcal T}_R(\tilde \gamma)}$ has already been treated.

We first treat the case where either $|\prtl_{r}\phi(r,s)| \geq e^{-4T}$ is satisfied for all $(r,s) \in [0,1/4]^2$ or $|\prtl_{s}\phi(r,s)| \geq e^{-4T}$ for all $(r,s) \in [0,1/4]^2$, and here the first term in parentheses in \eqref{kalphabd3D} is not needed.  We may assume the former, as the latter case can be treated by taking adjoints of the oscillatory integral operators.  We first extend Lemma \ref{T:mixedpartial}, at which point \eqref{kalphabd3D} follows by applying Lemmas \ref{T:expupperbd} and \ref{T:oiolemma} (which still applies over $[0,1/4]$).
\begin{lemma}\label{T:mixedpartial3D}
Let $n=3$, $\phi(r,s) = d(\gamma(s),\eta(r))$ be as in \S\ref{S:variation}. Suppose $\wtM$ has constant sectional curvatures equal to $-1$.  Assume $(r_0,s_0)\in [0,\frac 14]^2$ satisfies
\begin{equation}\label{firstpartiallower}
|\prtl_{r}\phi(r_0,s_0)| \geq e^{-4T}.
\end{equation}
Denoting $\rho_0 = d(\gamma(s_0),\eta(r_0))$, suppose suppose further that $3 \leq \rho_0 \leq T$ and $\angle (\nabla_1 d(\gamma(s_0),\eta(r_0)),\pm\dot\gamma(s_0)) \geq e^{-T}$ for both choices of $\pm$.  There exists uniform constants $C_i>0$, $i=1,2$ such that for $T$ sufficiently large
\begin{equation}\label{thirdpartiallower3D}
|\prtl_{rss}^3\phi(r_0,s_0)| \geq e^{-C_1T} .
\end{equation}
whenever
\begin{equation}\label{secondpartialupper3D}
|\prtl_{rs}^2\phi(r_0,s_0)| \leq e^{-C_2T}.
\end{equation}
\end{lemma}

\begin{proof}
Using the dimensionless bounds \eqref{deeteedublower} and \eqref{accelerationbd}, we have that
\[
1 \geq |D_tW_0^\perp|_g \geq e^{-2T} \quad\text{ and } \quad|(D_s^2\prtl_t \Psi|_{(r_0,s_0,0)})^\perp|_g \lesssim e^{CT}
\]
respectively.  Our main claim is that for some larger value of $C$,
\begin{equation}\label{accelerationbd3D}
\left|\langle V_0^\perp, (D_s^2\prtl_t \Psi)^\perp|_{(r_0,s_0,0)}\rangle\right|
\lesssim
e^{CT}|\prtl_{rs}^2\phi(r_0,s_0)|.
\end{equation}
This was the key inequality in the proof of Lemma \ref{T:mixedpartial}, where the rigidity of 2 dimensions allowed us to argue that upper bounds on $|\langle V_0^\perp, D_t W_0^\perp\rangle|$ implied bounds on $|V_0^\perp|_g$.  Here we use that $D_tW_0$ and $(D_s^2\prtl_t \Psi)^\perp|_{(r_0,s_0,0)}$ are tangent to the same 2 dimensional submanifold to reason the same way.  It is this step that uses the constant curvature hypothesis in an essential way and seems to be a considerable obstacle to establishing Theorem \ref{T:maintheorem3D} in the more general setting of nonconstant nonpositive curvature. Note that
\[
|\prtl_{r}\phi(r_0,s_0)| = |\langle \nabla_1 d(\eta(r_0),\gamma(s_0)), \dot\eta(r_0)\rangle| = |\langle V_0, \dot\sigma_0\rangle |.
\]
So once \eqref{accelerationbd3D} is established, we have that by \eqref{thirdvariation} and the hypothesis \eqref{firstpartiallower}
\[
|\prtl_{rss}^3\phi(r_0,s_0)| \gtrsim (\log\lambda)^{-4} e^{-4T}-e^{CT}|\prtl_{rs}^2\phi(r_0,s_0)| = e^{-8T}-e^{CT}|\prtl_{rs}^2\phi(r_0,s_0)|
\]
for some $C$ large enough.  The lemma then follows by taking $C_2$ sufficiently large relative to the exponential constant here.

To see \eqref{accelerationbd3D}, we use an argument which is a coordinate free variation on one appearing in \cite[p.453-4]{ChenSogge}.  Let $N$ be the image of the 2 dimensional subspace span$(\dot\gamma(s_0),\dot\sigma(\rho_0))$ under the exponential map at acting on the tangent space at $\gamma(s_0) =\sigma(\rho_0)$.  Since $(\wtM,\tilde g)$ is hyperbolic space, it is known $N$ is a 2 dimensional totally geodesic submanifold of $\wtM$.  But this means that the geodesic triangle joining $\eta(r_0)$, $\gamma(s)$, and $\sigma(\rho_0)$ must lie in $N$.  In other words, $\Psi(r_0,s,t)\in N$, which implies that both $\prtl_s \Psi(r_0,s,t)$ and $\prtl_t \Psi(r_0,s,t)$ are tangent to $N$ for $(s,t) \in [0,1]\times [0,\rho_0]$.  Moreover, since the second fundamental form of $N$ vanishes, $(D_s^2\prtl_t \Psi)(r_0,s,t)$ is also tangent to $N$.  Now let $Q_t, W_t^\perp$ be as in \eqref{paralleljacobi}.  The previous observation implies that $W_t^\perp$ is tangent to $N$ and hence $Q_t$ is tangent to $N$ for every $t \in [0,\rho]$.  Hence $Q_t$, $\dot\sigma_t$ are orthonormal vectors that span the tangent space to $N$ at every point $\sigma(t)$.  This implies that $(D_s^2\prtl_t \Psi)^\perp = |(D_s^2\prtl_t \Psi)^\perp|_gQ_t$ and hence
\[
|\langle V_0^\perp, (D_s^2\prtl_t \Psi|_{(r_0,s_0,0)})^\perp\rangle| = |\langle V_0^\perp, Q_0\rangle|\,
|(D_s^2\prtl_t \Psi|_{(r_0,s_0,0)})^\perp|_g.
\]
Since $D_tW_0^\perp = |D_tW_0^\perp|_g Q_0$, we also have
\[
|\prtl_{rs}^2\phi(r_0,s_0)| = |\langle V_0^\perp, D_t W_0^\perp\rangle| = |\langle V_0^\perp, Q_0\rangle|\,
|D_tW_0^\perp|_g.
\]
We can thus conclude that
\[
|\langle V_0^\perp, (D_s^2\prtl_t \Psi|_{(r_0,s_0,0)})^\perp\rangle| =
\frac{|\prtl_{rs}^2\phi(r_0,s_0)|\,|(D_s^2\prtl_t \Psi|_{(r_0,s_0,0)})^\perp|_g}{|D_tW_0^\perp|_g}
\]
which implies \eqref{accelerationbd3D} once we recall the exponential upper and lower bounds on $|(D_s^2\prtl_t \Psi|_{(r_0,s_0,0)})^\perp|_g$, $|D_tW_0^\perp|_g$ respectively.
\end{proof}

We now consider the case where
\begin{equation}\label{isolatedcase}
|\prtl_{s}\phi(r_1,s_1)|, |\prtl_{r}\phi(r_1,s_1)| \leq e^{-4T} \qquad \text{ for some }(r_1,s_1) \in [0,1/4]^2.
\end{equation}
The first observation is that such points are in some sense isolated in $[0,1/4]^2$, which follows by the following lemma which is symmetric in $(r,s)$.  In \cite{ChenSogge}, the authors showed that the joint zeros of $|\prtl_{s}\phi(r_1,s_1)|= |\prtl_{r}\phi(r_1,s_1)|=0$ are unique by arguing that 2 distinct joint zeros, give rise to a geodesic quadrilateral with all four angles equal to $\pi/2$, which is impossible in the presence of negative curvature.  Here we must work quantitatively, considering the case where $|\prtl_{r}\phi|$, $|\prtl_{s}\phi|$ are both small, but not necessarily vanishing.  We instead use a somewhat crude convexity argument to show this principle.  It is this lemma that uses the rather artificial restriction to $(r,s)\in[0,1/4]^2$.

\begin{lemma}\label{T:isolatedlemma}
Suppose $(r_1,s_1) \in [0,1/4]^2$ satisfies
\begin{equation*}
|\prtl_{r}\phi(r_1,s_1)| \leq e^{-4T} = (\log \lambda)^{-4}
\end{equation*}
for some $T=\log\lambda$ sufficiently large.  Suppose further that $\phi(r,s)$ is sufficiently large in $[0,1/4]^2$ in that $\sinh \phi(r,s) \geq  64$ and $\coth \phi(r,s)\leq 3/2$. Then if $(r,s) \in [0,1/4]^2$ satisfies
\begin{equation}\label{outsidebadcube}
|r-r_1| \geq \max\left(\frac 12 |s-s_1|,(\log \lambda)^{-3}\right)
\end{equation}
we have the lower bound
\begin{equation}\label{phirlower}
|\prtl_{r}\phi(r,s)| \gtrsim e^{-3T}=(\log \lambda)^{-3}.
\end{equation}
\end{lemma}
Thus if \eqref{isolatedcase} is satisfied, we let $\tilde{a}_{\pm,\alpha}$ be the product of $a_{\pm,\alpha}$ with a bump function such that $\tilde{a}_{\pm,\alpha}$ vanishes in a cube of sidelength $2(\log \lambda)^{-3}$ about $(r_1,s_1)$ and
\begin{equation}\label{amplitudeC1}
\|\tilde{a}_{\pm,\alpha}\|_{C^0}+(\log \lambda)^{-3}\|\nabla \tilde{a}_{\pm,\alpha}\|_{C^0} \lesssim 1.
\end{equation}
Setting
\[
\tilde{K}_\alpha(r,s) := \frac{\lambda}{T\phi(r,s)}\sum_{\pm}\tilde{a}_{\pm,\alpha}(r,s)e^{\pm i\lambda\phi(r,s)},
\]
we have
\[
\int_0^{\frac 14} |K_\alpha(r,s)-\tilde{K}_\alpha(r,s)|\,dr + \int_0^{\frac 14} |K_\alpha(r,s)-\tilde{K}_\alpha(r,s)|\,ds \lesssim \frac{\lambda}{(\log\lambda)^3}
\]
so that the contribution of $K_\alpha-\tilde{K}_\alpha$ to \eqref{kalphabd3D} is bounded by the right hand side there and we are left to prove \eqref{kalphabd3D} with $K_\alpha$ replaced by $\tilde{K}_\alpha$.

We now write $\tilde{a}_{\pm,\alpha}$ as the sum of two functions supported in $|r-r_1| \geq 2 |s-s_1|$ and $|s-s_1| \geq 2 |r-r_1|$ respectively.  Given the initial localization, this can be done so that \eqref{amplitudeC1} is still satisfied.  By taking adjoints, we are reduced to treating the component satisfying the former support condition. Then Lemmas \ref{T:mixedpartial3D} and \ref{T:oiolemma} apply to the resulting oscillatory integral operator, which concludes the proof of \eqref{kalphabd3D} as the $C^1$ norm of the new amplitude introduces an acceptable loss of $O((\log\lambda)^3) = O(e^{3T})$ in the estimates.

\begin{proof}[Proof of Lemma \ref{T:isolatedlemma}]
Observe that
\begin{align}
\prtl_{rr}^2\phi(r,s) &= |\dot\eta(r)^\perp|^2_g\coth d(\eta(r),\gamma(s)),\label{2ndpartialr}\\
|\prtl_{rs}^2\phi(r,s)| &\leq |\dot\eta(r)^\perp|_g |\dot\gamma(r)^\perp|_g\, \csch d(\eta(r),\gamma(s)),\label{2ndpartials}
\end{align}
where the $\perp$ denotes the component orthogonal to the geodesic joining $\eta(r)$ and $\gamma(s)$.  These can be verified at a point $(r_0,s_0)$ by inserting the identities in \eqref{paralleljacobi} into
\eqref{2ndpartialr0} and \eqref{secondvariation}.  We will first prove the lower bound \eqref{phirlower} under the assumption that $|\dot\eta(r)^\perp|^2_g \geq 1/16$, then justify this below.  Given all this, a Taylor expansion shows that for some $(\tilde r, \tilde s)\in [0,1/4]$,
\begin{multline*}
|\prtl_{r}\phi(r,s)| \geq  |r-r_1||\dot\eta(\tilde r)^\perp|^2_g\,\coth d(\eta(\tilde r),\gamma(\tilde s)) - \\
|s-s_1||\dot\eta(\tilde r)^\perp|_g |\dot\gamma(\tilde r)^\perp|_g \,\csch d(\eta(\tilde r),\gamma(\tilde s)) -|\prtl_{r}\phi(r_1,s_1)|,
\end{multline*}
Since we assume that $\sinh d(\eta(\tilde r),\gamma(\tilde s)) \geq 64$ and $\coth d(\eta(\tilde r),\gamma(\tilde s)) \leq 3/2$
\begin{equation*}
|\prtl_{r}\phi(r,s)| \geq  \frac{1}{16}|r-r_1| -
\frac 1{64}|s-s_1| - e^{-4T}  \geq \frac 1{32}|r-r_1|-e^{-4T} .
\end{equation*}
The lower bound \eqref{phirlower} now follows for large $T$ when $|r-r_1|\geq (\log\lambda)^{-3}$.

To see that $|\dot\eta(r)^\perp|^2_g \geq 1/16$, first observe that since we also assume that
$\coth d(\eta(r),\gamma(s)) \leq 3/2$, we have that by \eqref{2ndpartialr}, \eqref{2ndpartials},
\[
|\prtl_{r}\phi(r,s)| \leq \frac 32 |r-r_1| + \frac 1{64} |s-s_1| + |\prtl_{r}\phi(r_1,s_1)| \leq \frac 12 +e^{-4T}\leq \frac{\sqrt{15}}{4}.
\]
This is enough since
$
|\dot\eta(r)^\perp|^2_g = 1-\langle \nabla_1 d(\eta(r),\gamma(s)),\dot{\eta}(r)\rangle^2 = 1-(\prtl_{r}\phi(r,s))^2.
$
\end{proof}

\bibliographystyle{amsalpha}
\bibliography{bibtexdata}
\end{document}